\documentclass[article,11pt,reqno]{amsart}
\usepackage{etoolbox}

\makeatletter
\patchcmd{\@startsection}
  {\@afterindenttrue}
  {\@afterindentfalse}
  {}{}
\makeatother

\usepackage{latexsym,array,delarray,amsthm,amssymb}
\usepackage{enumerate}
\usepackage{tikz}

\usepackage[colorinlistoftodos]{todonotes} 
\usepackage[hidelinks]{hyperref}

\theoremstyle{plain}
\newtheorem{thm}{Theorem}[section]
\newtheorem{lemma}[thm]{Lemma}
\newtheorem{prop}[thm]{Proposition}

\newtheorem{conj}[thm]{Conjecture}
\newtheorem*{thm*}{Theorem}
\newtheorem*{lemma*}{Lemma}
\newtheorem*{prop*}{Proposition}
\newtheorem*{cor*}{Corollary}
\newtheorem*{conj*}{Conjecture}

\theoremstyle{definition}

\newtheorem{ex}[thm]{Example}

\newtheorem{ques}[thm]{Question}

\theoremstyle{remark}

\newcommand{\Z}{\mathbb{Z}}
\newcommand{\N}{\mathbb{N}}

\newcommand{\R}{\mathbb{R}}

\newcommand{\TT}{\mathbb{T}}
\newcommand{\GL}{\mathrm{GL}}

\newcommand{\poly}{\mathcal{P}({\Z^d})}

\newcommand{\polyt}{\mathcal{P}({\Z^2})}

\def\OL{\operatorname{L}}
\def\MM{\operatorname{M}}

\DeclareMathOperator*{\aff}{aff}

\newcommand{\conv}{\mathrm{conv}}

\newcommand{\interior}[1]{#1^\circ}

\newcommand*\diff{\mathop{}\!\mathrm{d}}

\title{Ehrhart tensor polynomials}

\author{S\"oren Berg}
\address{Institut f\"ur Mathematik, Technische Universit\"at Berlin, %
Germany}
\email{berg@math.tu-berlin.de}
\author{Katharina Jochemko}
\address{Department of Mathematics, %
Royal Institute of Technology, Stockholm, Sweden}
\email{jochemko@kth.se}
\author{Laura Silverstein}
\address{Institut f\"ur Diskrete Mathematik und Geometrie, %
Technische Universit\"at Wien, %
Austria}
\email{laura.silverstein@tuwien.ac.at}

\keywords{Ehrhart tensor polynomial, $h^r$-tensor polynomial, Pick's formula, positive semidefinite coefficients, half-open polytopes}
\subjclass[2010]{05A10, 05A15, 15A45, 15A69, 52B20, 52B45}

\date{\today}

\begin{document}
\maketitle

\begin{abstract}
The notion of Ehrhart tensor polynomials, a natural generalization of the Ehrhart polynomial of a lattice polytope, was recently introduced by Ludwig and Silverstein. We initiate a study of their coefficients. In the vector and matrix cases, we give Pick-type formulas in terms of triangulations of a lattice polygon. As our main tool, we introduce $h^r$-tensor polynomials, extending the notion of the Ehrhart $h^\ast$-polynomial, and, for matrices, investigate their coefficients for positive semidefiniteness. In contrast to the usual $h^\ast$-polynomial, the coefficients are in general not monotone with respect to inclusion. Nevertheless, we are able to prove positive semidefiniteness in dimension two. Based on computational results, we conjecture positive semidefiniteness of the coefficients in higher dimensions. Furthermore, we generalize Hibi's palindromic theorem for reflexive polytopes to $h^r$-tensor polynomials and discuss possible future research directions.
\end{abstract}

\section{Introduction}
The Ehrhart polynomial of a lattice polytope counts the number of lattice points in its integer dilates and is arguably the most fundamental arithmetic invariant of a lattice polytope. It is a cornerstone of geometric combinatorics and appears in various guises in other areas of mathematics such as commutative algebra, optimization, representation theory, or voting theory (see, e.g., \cite{Barvinok,Berenstein,DeLoeraetal,Lepelley,MillerSturmfels}). Concepts from Ehrhart theory have been generalized in various directions; for example, $q$-analogs of Ehrhart polynomials \cite{chapoton2016q}, equivariant versions \cite{Stapledonequivariant}, multivariate extensions  \cite{beck2001multidimensional,bihan2016irrational,haase2015mixed}, and generalizations to valuations \cite{jochemkocombinatorial,jochemko2016combinatorial,McMullen77}. 

Recently, Ludwig and Silverstein \cite{LS} introduced Ehrhart tensor polynomials based on discrete moment tensors that were defined by B\"or\"oczky and Ludwig~\cite{boroczkyvaluations}. Let $\poly$ denote the family of convex polytopes with vertices in $\Z^d$ called \textbf{lattice polytopes} and let $\TT^r$ be the vector space of symmetric tensors of rank $r$ on $\R^d$. For $x,y \in \mathbb{R}^d$, we write $xy$ for $x\otimes y$. In particular, $x^r=x\otimes \cdots \otimes x$ and we set $x^0 :=1$. 

\goodbreak
The \textbf{discrete moment tensor of rank $r$} of a polytope $P\in\poly$ is
\begin{equation}\label{eq:discretemoment}
\OL^{r}(P) \ = \ \sum_{x\in P\cap\Z^d}x^{r} 
\end{equation}
where $r\in\N$ and $\N$ denotes the set of nonnegative integers.
Note that, for our convenience, this definition differs by a scalar from the original definition given in~\cite{boroczkyvaluations}. A version of $\OL^{r}(P)$, the discrete directional moment, was studied in~\cite{Schulz}.
For $r=0$, the usual \textbf{discrete volume} or \textbf{lattice point enumerator} $\OL(P):=\OL^0 (P) = |P\cap \Z^d|$ is recovered. For $r=1$, $\OL^1(P)$ equals the discrete moment vector defined in~\cite{boroczky2016minkowski}. Based on results by Khovanski{\u\i} and Pukhlikov~\cite{PK92} and Alesker\cite{Alesker98}, it was identified in~\cite{LS} that $\OL^r (nP)$ is given by a polynomial, for any $n\in\N$, extending Ehrhart's celebrated result for the lattice point enumerator~\cite{Ehrhart62}.

\begin{thm*}[{\cite[Theorem 1]{LS}}]
There exist $\OL _i ^r \colon \mathcal{P}(\Z ^d) \rightarrow \TT^r$ for all $1\leq i\leq d+r$ such that
\[
\OL ^r(nP) \ = \ \sum _{i=0}^{d+r} \OL _i ^r (P)n^i
\]
for any $n\in\N$ and $P\in \mathcal{P}(\Z ^d)$.
\end{thm*}
\noindent The expansion of $\OL^r(nP)$ will be denoted as $\OL_P^r(n)$ and is called the \textbf{Ehrhart tensor polynomial} of $P$ in commemoration of this result. Furthermore, the coefficients $\OL^r_1,\dots,\OL^r_{d+r}$ are the \textbf{Ehrhart tensor coefficients} or \textbf{Ehrhart tensors}.

A fundamental and intensively studied question in Ehrhart theory is the characterization of Ehrhart polynomials and their coefficients. 
The only coefficients that are known to have explicit geometric descriptions are the leading, second-highest, and constant coefficients for the classic Ehrhart polynomial (see, e.g.,~\cite{ccd}). For the Ehrhart tensor polynomial, the leading and constant coefficients were given in~\cite{LS} and we give an interpretation for the second-highest coefficient (Proposition~\ref{prop:secondcoeff}) as the weighted sum of moment tensors over the facets of the polytope; the descriptions of all are given in Section~\ref{sec:discretemoment}. 

Conversely, for lattice polygons, the coefficients of the Ehrhart polynomial are positive and well-understood. They are given by Pick's Formula~\cite{Pick}. Let $\partial P$ denote the boundary of the polytope $P$, for any $P\in\poly$.

\begin{thm*}[Pick's Formula]
For any lattice polygon $P$, we have
\[
\OL(nP) \ = \ \OL _0(P)+\OL _1(P)n+\OL _2(P)n^2 \, 
\]
where $\OL _0(P)=1$, $\OL _1(P)=\tfrac{1}{2}\OL(\partial P)$, and $\OL _2(P)$ equals the area of $P$.
\end{thm*}

In Section~\ref{sec:pick}, we determine Pick-type formulas for the discrete moment vector and matrix. Our interpretation of the coefficients is given with respect to a triangulation of the respective polygon. The principal tool we use to study Ehrhart tensor polynomials are \textbf{$h^r$-tensor polynomials} which encode the Ehrhart tensor polynomial in a certain binomial basis. Extending the notion of the usual Ehrhart $h^\ast$-polynomial, we consider
\begin{equation}
\OL^r(nP) \ = \ h^{r}_0 (P){n+d+r \choose d+r}+h^{r}_1 (P){n+d+r-1 \choose d+r}+ \ \cdots \ + h^{r}_{d+r} (P){n \choose d+r} \,
\end{equation}
for a $d$-dimensional lattice polytope $P$
and define the \textbf{$h^r$-tensor polynomial} of $P$ to be
\[
h^{r}_P (t)\ = \ \sum _{i=0}^{d+r}h^{r}_i(P)t^i.
\]

\goodbreak
We determine a formula for the $h^r$-tensor polynomial of half-open simplices (Theorem~\ref{prop:h_cone}) by using \textbf{half-open decompositions of polytopes}; an important tool which was introduced by K\"oppe and Verdoolaege~\cite{koppeverdoolaege}. From this formula and the existence of a unimodular triangulation, we deduce an interpretation of all Ehrhart vectors and matrices of lattice polygons.

Stanley's Nonnegativity Theorem \cite{RS80} is a foundational result which states that all coefficients of the $h^\ast$-polynomial of a lattice polytope are nonnegative. Stanley moreover proved that the coefficients are monotone with respect to inclusion; that is, for all lattice polytopes $Q \subseteq P$ and all $0\leq i\leq d$, it holds that $h^{\ast}_i(Q)\leq h^{\ast}_i(P)$. Using half-open decompositions, it was proven in \cite{jochemkocombinatorial} that, with regard to translation invariant valuations, monotonicity and nonnegativity are equivalent. 

In Section~\ref{sec:positivity}, we discuss notions of positivity for Ehrhart tensors and investigate Ehrhart tensor polynomials and $h^2$-tensor polynomials with respect to \textbf{positive semi\-definiteness}. In contrast to the usual Ehrhart polynomial, Ehrhart tensor coefficients can even be negative definite for lattice polygons (Example~\ref{ex:nonpositiveEhrharttensors}). Moreover, the coefficients of $h^2$-tensor polynomials are not monotone which is demonstrated by Example~\ref{ex:counter_ex_monotonicity}. Therefore, techniques such as irrational decompositions and half-open decompositions that have been used to prove Stanley's Nonnegativity Theorem (see \cite{ccd,jochemkocombinatorial}) can not immediately be applied to $h^2$-tensor coefficients. Nevertheless, considering an intricate decomposition of lattice points inside a polygon, we are able to prove positive semi-definiteness of the coefficients of $h^2$-tensor polynomial in dimension two (Theorem~\ref{thm:main}). We remark here that the theorem holds true for lattice polygons in a higher dimensional ambient space. Furthermore, all of the results given in this article are independent of the ambient space. Based on computational results, we further conjecture positive-semidefiniteness of the $h^2$-tensor coefficients in higher dimensions (Conjecture~\ref{conj:pos}).

In Section~\ref{sec:further}, we prove a generalization of Hibi's Palindromic Theorem~\cite{Hibi91} characterizing reflexive polytopes as having palindromic $h^r$-tensor polynomials for $r\in\N$ of even rank and conclude by discussing possible future research directions.

\goodbreak
\section{Discrete moment tensors}\label{sec:discretemoment}
We introduce some general notions we will use here yet assume basic knowledge of polyhedral geometry and, in particular, lattice polytopes. For further reference, we recommend~\cite{ccd,ziegler}.

\goodbreak
We work in $d$-dimensional Euclidean space, $\R^d$, equipped with the scalar product $u\cdot v$, for any $u,v\in\R^d$. The vector space of symmetric tensors $\TT^r$ is then canonically isomorphic to the space of multi-linear functionals from $(\R^d)^r$ to $\mathbb{R}$ that are invariant with respect to permutations of the arguments. We have $\TT^0=\R$ and can now identify $\TT^1$ with $\R^d$. Given the standard orthonormal basis $e_1,\dots,e_d$, any tensor $T\in\TT^r$ can be written uniquely as 
\[
T\ = \ \sum_{1\leq i_j\leq d}T_{i_1\dots i_r}e_{i_1}\otimes\dots\otimes e_{i_r}.
\]
For $r=2$, the bilinear form $T\in\TT^2$ can then be identified with a symmetric $d\times d$ matrix $T=(T_{ij})$. To that end, we will call the discrete moment tensor~(\ref{eq:discretemoment}) of ranks 1~and~2 the \textbf{discrete moment vector} and \textbf{discrete moment matrix}, respectively. We will also regard their associated coefficients, their Ehrhart tensors, as \textbf{Ehrhart vectors} and \textbf{Ehrhart matrices}.

\goodbreak
Prior to describing the known Ehrhart tensors, we provide some properties of the discrete moment tensor that we will need. Considering $\OL^r$ with respect to its coordinates, for any $P\in\poly$, gives
\begin{equation*}
  \OL^r(P) (e_{i_1}, \dots, e_{i_r})\ = \ \sum_{x\in P\cap\Z^n} (x\cdot e_{i_1}) \cdots (x\cdot e_{i_r}).
\end{equation*}
\noindent Hence the action of $\GL (\Z^d)$, the general linear group over the integers, on $\OL^r$ is observed to be
\begin{equation*}
  \OL^r(\phi P) (e_{i_1}, \dots, e_{i_r})\ = \ \OL^r(P) (\phi^t e_{i_1}, \dots, \phi^t e_{i_r})
\end{equation*}
for any $P\in\poly$ and $\phi\in\GL (\Z^d)$; we say that $\OL^r$ is \textbf{$\GL (\Z^d)$ equivariant}.

We let $P^o$ denote the relative interior of $P$ with respect to its affine hull, denoted by $\aff (P)$, and for any $P\in\poly$ and $r\in\N$ we set 
\[
\OL^r(P^o)\ := \ \sum _{x\in P^o} x^r.
\]
For the discrete volume, the Ehrhart-Macdonald reciprocity was a fundamental result in Ehrhart theory that was established by Ehrhart~\cite{Ehrhart62} and first proven by Macdonald~\cite{Macdonald71}. 

\begin{thm}{\cite{Ehrhart62,Macdonald71}}
If $P$ is a $d$-dimensional lattice polytope, then
\[
\OL(nP^o)\ = \ (-1)^d\OL_P(-n).
\]
\end{thm}
\noindent A general version of this result was given for translation invariant valuations by McMullen~\cite{McMullen77}. Unlike the discrete volume, the discrete moment tensor varies under translations by elements in~$\Z^d$. More precisely, for all $r\in\N$, the discrete moment tensor of a translated polytope is
\[
\OL^{r}(P+t) \ = \ \sum _{j=0} ^r {r\choose j}\OL^{r-j}(P)t^j \, 
\]
and we say that the discrete moment tensor is covariant with respect to translations or \textbf{translation covariant}. A \textbf{unimodular transformation} of a polytope $P\in\poly$ is a $\GL (\Z^d)$ transformation of $P$ paired with a translation.

Similar to McMullen, a reciprocity theorem was given for translation covariant valuations in~\cite{LS}. Extending the classical Ehrhart-Macdonald reciprocity, the following reciprocity theorem gives the special case of the discrete moment tensor. 

\begin{thm}{\cite[Theorem 2]{LS}}\label{thm:tensorreciprocity}
Let $P$ be lattice polytope. Then 
\[
\OL ^r_P  (-n) \ = \ (-1)^{\dim (P)+r}\OL^r (nP^o) \, .
\]
\end{thm}
\noindent We use this theorem in our characterization of the second-highest Ehrhart tensor.

A complete characterization of the Ehrhart coefficients has been inaccessible up to this point. The coefficients can even be negative and, therefore, are difficult to describe combinatorially. However, it is known that the leading coefficient equals the volume, the second highest coefficient is related to the normalized surface area, and the constant coefficient is always $1$. 

\goodbreak
More generally, for Ehrhart tensors, it has been proven \cite[Lemma 26]{LS} that the leading coefficient of the discrete moment tensor equals the \textbf{moment tensor of rank $r$} which is defined as
\[
\MM^r (P) \ = \ \int _P x^r \, \diff x \, .
\]
It is also clear that, for $r\geq 1$, the constant coefficient vanishes identically by its $\GL (\Z^d)$ equivariance; that is, $\OL_0^r(P)=\OL^r(0P)=0$ for any $P\in\poly$~\cite{LS}.

We give an interpretation for the second coefficient (Proposition~\ref{prop:secondcoeff}) as the weighted sum of moment tensors over the facets of the polytope. The coefficient $\OL_{d-1}(P)$, specifically, was shown to be equal to one half of the sum over the normalized volumes of the facets of $P$ by Ehrhart~\cite{Ehrhart67}. We extend this statement to Ehrhart tensor polynomials by proving the following. 

\begin{prop}\label{prop:secondcoeff}
Let $P$ be a lattice polytope. Then
$$
\OL^r_{\dim (P)+r-1}(P) \ = \ \sum _{F} \frac{1}{|\det (\aff (F)\cap \mathbb{Z}^d)|}\int _F x^r \, \diff x \, ,
$$
where the sum is over all facets $F \subset P$.
\end{prop}

\goodbreak
\begin{proof}
Theorem~\ref{thm:tensorreciprocity}, on the one hand, implies
\[
\sum _{x\in \partial nP} x^r \ = \ \sum _{F\subsetneq P} \sum _{x\in nF^o} x^r \ = \ \sum _{F\subsetneq P} (-1)^{\dim (F) +r}\OL ^r_F (-n) \, ,
\]
where the sum is taken over all proper faces $F\subsetneq P$. On the other hand, we have
\begin{align*}
\sum _{x\in \partial nP} x^r\  &= \ \OL^r (nP)-\OL^r (nP^o)
\ = \ \OL^r (nP)-(-1)^{\dim (P) +r}\OL^r_P  (-n)\\
&= \ 2\sum _{i\geq 0}\OL^r _{\dim (P) +r -1-2i} (nP) \, 
\end{align*}
where we set $\OL^r_i=0$ for all $i<0$. Using both equations, we obtain
\begin{align*}
\OL ^r _{\dim (P) +r-1} (P)\ &= \ \lim _{n\to \infty} \frac{1}{n^{\dim (P) +r -1}} \sum _{i\geq 0}\OL^r _{\dim (P) +r -1-2i} (nP)\\
&= \ \frac{1}{2}\sum _{F\subsetneq P} (-1)^{\dim (F) +r}\lim _{n\to \infty} \frac{1}{n^{\dim (P) +r -1}}\OL ^r_F  (-n)\\
&= \ \frac{1}{2}\sum _{F \text{ facet}} \frac{1}{|\det (\aff (F)\cap \mathbb{Z}^d)|}\int _F x^r\, \diff x \, ,
\end{align*}
where the last equality follows from~\cite{LS}.
\end{proof}

\section{$h^r$-tensor polynomials}\label{chapter:h^r}
Let $P$ be a $d$-dimensional lattice polytope. Since $\OL ^r(nP)$ is a polynomial of degree at most $d+r$, it can be written as a linear combination of the polynomials ${n+d+r \choose d+r}, {n+d+r-1 \choose d+r},\ldots,{n \choose d+r}$, that is,
\begin{equation}\label{eq:expansion}
\OL^r(nP) \ = \ h^{r}_0 (P){n+d+r \choose d+r}+h^{r}_1 (P){n+d+r-1 \choose d+r}+ \ \cdots \ + h^{r}_{d+r} (P){n \choose d+r} \, 
\end{equation}
for some $h^{r}_0(P),\ldots,h^{r}_{d+r}(P)\in \mathbb{T}^r$. Equivalently, in terms of generating functions,
\begin{equation}\label{eq:h_cone}
\sum _{n\geq 0} \OL^r (nP)t^n \ = \ \frac{h^{r}_0(P)+h^r_1(P)t+\dots+h^r_{d+r}(P)t^{d+r}}{(1-t)^{d+r+1}} \, .
\end{equation}
We call $h^{r}(P)=(h^{r}_0(P),h^{r}_1(P),\ldots,h^{r}_{d+r}(P))$ the \textbf{$h^r$-vector}, its entries the \textbf{$h^r$-tensor coefficients} or \textbf{$h^r$-tensors} of $P$, and 
\[
h^{r}_P (t)\ = \ \sum _{i=0}^{d+r}h^{r}_it^i
\]
the \textbf{$h^r$-tensor polynomial} of $P$. Observe that for $r=0$ we obtain the usual $h^\ast$-polynomial and $h^\ast$-vector of an Ehrhart polynomial. By evaluating equation~\eqref{eq:expansion} at $n=0,1$, we obtain $h_0^{r}=0$ for $r\geq 1$ and $h_1^{r}=\OL^r (P)$ for $r\geq 0$. Inspecting the leading coefficient, we obtain
\[
h_1^{r}(P)+h_2^{r}(P)+\ldots+h_{d+r}^{r}(P)\ = \ (d+r)!\int_Px^r\mathrm{d}x \, .
\]
Applying Theorem~\ref{thm:tensorreciprocity} and evaluating at $n=1$, we obtain
\[
h^r_{d+r} (P)\ = \ \OL^r (P^o) \, .
\]
\subsection{Half-open polytopes}
We will not only consider relatively open polytopes, but also \textbf{half-open polytopes}. Let $P$ be a polytope with facets $F_1,\ldots,F_k$ and let $q$ be a generic point in its affine span $\aff (P)$. Then a facet $F_i$ is \textbf{visible} from $q$ if $(p,q]\cap P = \emptyset$ for all $p\in F$. If $I_q (P)=\{i\in [k] \colon F_i \text{ is visible from }q\}$ then the point set
\[
H_q (P) \ = \ P\setminus \bigcup_{i \in I_q (P)} F_i
\]
defines a half-open polytope. In particular, $H_q (P)=P$ for all $q\in P$. The following result by K\"oppe and Verdoolaege~\cite{koppeverdoolaege} shows that every polytope can be decomposed into half-open polytopes, and is implicitely also contained in works by Stanley and Ehrhart (see~\cite{RS74}).
\begin{thm}[\cite{koppeverdoolaege}]\label{thm:half-open decomposition}
Let $P$ be a polytope and let $P_1,\ldots,P_m$ be the maximal cells of a triangulation of $P$. Let $q\in \aff (P)$ be a generic point. Then
\[
H_q (P) \ = \ H_q (P_1)\sqcup H_q (P_2)\sqcup \cdots \sqcup H_q (P_m)
\]
is a partition.
\end{thm}
The discrete moment tensor naturally can be defined for half-open polytopes by setting
\[
\OL^r (H_q (P)) \ := \ \OL^r(P) - \sum _{J\subseteq I_q(P)} (-1)^{\dim P -\dim F_J} \OL^r (F_J) \, 
\]
where $F_J := \bigcap _{i \in J} F_i$. Then, from Theorem \ref{thm:half-open decomposition} and the inclusion-exclusion principle, we obtain that  
\begin{equation}\label{eq:halfopendec}
\OL^r (P) \ = \ \OL^r (H_q (P_1)) \ + \ \OL^r (H_q (P_2)) \ + \ \cdots \ + \  \OL^r (H_q (P_m)) \, 
\end{equation}
(Compare also~\cite[Corollary 3.2]{jochemkocombinatorial}).
\subsection{Half-open simplices}\label{sec:half-open}
Let $S$ be a $d$-dimensional lattice simplex with vertices $v_1,\ldots, v_{d+1}$. Let $F_1,\ldots, F_{d+1}$ denote the facets of $S$ such that $v_i \not \in F_i$. Let $S^{\ast}=H_q (S)$ be a $d$-dimensional half-open simplex and let $I=I_q (S)$. We define the half-open polyhedral cone
\[
C_{S^{\ast}} \ = \ \left\{ \sum _{i=1}^{d+1} \lambda _i \bar{v}_i : \lambda _i \geq 0\text{ for }i\in[d+1],\lambda_i\neq 0 \text{ if } i\in I\right\} \subseteq \mathbb{R}^{d+1}
\]
where $\bar{v}_i := (v_i, 1)\in\R^{d+1}$ for all $1\leq i\leq d+1$. Then, by identifying hyperplanes of the form $\{x\in \mathbb{R}^{d+1}\colon x_{d+1}=n\}$ with $\mathbb{R}^{d}$ via $p\colon \mathbb{R}^{d+1}\rightarrow \mathbb{R}^d$ which maps $x\mapsto (x_1,\ldots, x_d)$, we have $C_{S^{\ast}} \cap \{x_{d+1}=n\} = nS^{\ast}$. We consider the half-open parallelepiped
\[
\Pi _{S^{\ast}} \ = \ \left\{ \sum _{i=1}^{d+1} \lambda _i \bar{v}_i : 0< \lambda _i \leq 1 \text{ if } i\in I, 0\leq \lambda _i <1 \text{ if } i\not\in I\right\} \, .
\]
Then 
\[
C_{S^{\ast}} \ = \ \bigsqcup _{u\in \mathbb{Z}^{d+1}}\Pi _{S^{\ast}} +u_1\bar{v}_1+\cdots +u_{d+1}\bar{v}_{d+1} \, .
\]
Let $S_i = \Pi _{S^{\ast}} \cap \{x_{d+1}=i\}$. Then $S_i$ is a \textit{partially open} hypersimplex; that is, a hypersimplex with certain facets removed.

\goodbreak
Our next result shows that $\OL^r (nS^{\ast})$ is given by a polynomial in $n$ by determining its generating series. We follow the line of argumentation in~\cite[Proposition 3.3]{jochemkocombinatorial}. Observe that, together with equation \eqref{eq:halfopendec}, this reproves the polynomiality result of $\OL^r(nP)$.

\begin{thm}\label{prop:h_cone}
With the notation given above, the equation
\[
\sum _{n\geq 0} \OL^r (nS^{\ast}) t^n =  \sum _{k_0,\ldots, k_{d+1} \geq 0 \atop \sum k_j =r} {r\choose k_0,\ldots,k_{d+1}}v_1^{k_1}\cdots v_{d+1}^{k_{d+1}} \frac{(1-t)^{k _0}A_{k_1}(t)\cdots A_{k_{d+1}}(t)}{(1-t)^{d+r+1}}\sum _{i=0}^d\OL ^{k_0}(S_i)t^i \, ,
\]
holds true where $A_j(t)$ is the $j$-th Eulerian polynomial.
\end{thm}

\goodbreak
\begin{proof}
The generating function of the discrete moment tensor allows us to consider the discrete moment tensor of $nS^{\ast}$ by cutting the cone $C_{S^{\ast}}$ with the hyperplane $\{x_{d+1}=n\}$. The geometric interpretation of the half-open parallelepipeds tiling the cone, the translation covariance of the discrete moment tensor, and the binomial theorem together yield the equation
\begin{align*}
\sum _{n\geq 0} &\OL ^r(nS^{\ast})t^n  =\sum _{i=0}^d t^i \sum _{u_1,\ldots, u_{d+1} \geq 0} \OL ^r (S_i +u_1\bar{v}_1+\cdots +u_{d+1}\bar{v}_{d+1})t^{u_1+\cdots +u_{d+1}}\\
&=\sum _{i=0}^d t^i \sum _{u_1,\ldots, u_{d+1} \geq 0}\sum _{j=0}^r{r\choose j}\OL ^{r-j}(S_i) (u_1\bar{v}_1+\cdots +u_{d+1}\bar{v}_{d+1})^jt^{u_1+\cdots +u_{d+1}}\\
&=\sum _{i=0}^d t^i \sum _{u_1,\ldots, u_{d+1} \geq 0}\sum _{k_0,\dots, k_{d+1}\geq 0\atop \sum k_j =r} {r\choose k_0,\ldots,k_{d+1}}\OL ^{k_0}(S_i) (u_1\bar{v}_1)^{k_1}\cdots (u_{d+1}\bar{v}_{d+1})^{k_{d+1}}t^{u_1+\cdots +u_{d+1}}\\
&=\sum _{i=0}^d t^i \sum _{k_0,\dots, k_{d+1}\geq 0\atop \sum k_j =r} {r\choose k_0,\ldots,k_{d+1}}\OL ^{k_0}(S_i) \bar{v}_1^{k_1}\cdots \bar{v}_{d+1}^{k_{d+1}}\sum _{u_1,\ldots, u_{d+1} \geq 0}u_1^{k_1}\cdots u_{d+1}^{k_{d+1}}t^{u_1+\cdots +u_{d+1}}
\end{align*}
from which the result follows since
\[
\sum _{n\geq 0}n^jt^n \ = \ \frac{A_j(t)}{(1-t)^{j+1}} \, ,
\]
a known identity of generating functions (see, e.g.,~\cite{ccd}).
\end{proof}

We remark that the results and proofs of this section immediately carry over to general translative polynomial valuations (see~\cite{LS} for a definition). In particular,  Theorem~\ref{prop:h_cone} can be generalized to give a new proof of \cite[Corollary 5]{PK92}.

\goodbreak
\section{Pick-type formulas} \label{sec:pick} 
Pick's Theorem~\cite{Pick} gives an interpretation for the coefficients of the Ehrhart polynomial of a lattice polygon which establishes a relationship between the area of the polygon, the number of lattice points in the polygon and on its boundary. An analogue in higher dimensions can not exist (see, e.g.,~\cite{gruber}) as it is crucial that every polygon in dimension two has a unimodular triangulation; that is, a triangulation into simplices of minimal possible area $1/d!$. We offer interpretations for the coefficients of the Ehrhart tensor polynomial in the vector and the matrix cases by taking the route over the $h^r$-tensor polynomial. 

\goodbreak
Given a polygon $P\in\polyt$, we will consider unimodular triangulations of $P$ where such a triangulation will always be denoted by $\mathcal{T}$. The triangulation will be described by the edge graph $G=(V,E)$ of $\mathcal{T}$ where $V$ are the lattice points contained in $P$ and $E$ the edges of $\mathcal{T}$. Furthermore, the notation $x$ will be reserved for elements of $V$ and $y,z$ for endpoints of the edge $\{y,z\}\in E$. We define $V^o=P^o\cap\Z^2$, $\partial V=\partial P\cap\Z^2$, $E^o=\{\{y,z\}\in E : (y,z)\not\subset\partial P\} $, and $\partial E=\{\{y,z\}\in E : (y,z)\subset\partial P\} $.

\begin{figure}
    \begin{tikzpicture}
    \clip (-0.5,-0.65) rectangle (8,2);

    \coordinate (v_one) at (0,0);
    \coordinate (v_two) at (1,0);
    \coordinate (v_three) at (0,1);
    \draw [thin, fill=gray, fill opacity=0.2] (v_two) -- (v_one) -- (v_three) -- (v_two);
    \draw (v_one) node [below left] {$v_1$};
    \draw (v_two) node [below right] {$v_2$};
    \draw (v_three) node [above left] {$v_3$};
    \node[] at (0.5,-0.5) {$T_0$};

    \coordinate (v_one) at (3,0);
    \coordinate (v_two) at (4,0);
    \coordinate (v_three) at (3,1);
    \draw [thin, fill=gray, fill opacity=0.2] (v_two) -- (v_one) -- (v_three);
    \draw [dashed, fill opacity=0.2] (v_two) -- (v_three);
    \draw (v_one) node [below left] {$v_1$};
    \draw (v_two) node [below right] {$v_2$};
    \draw (v_three) node [above left] {$v_3$};
    \node[] at (3.5,-0.5) {$T_1$};

    \coordinate (v_one) at (6,0);
    \coordinate (v_two) at (7,0);
    \coordinate (v_three) at (6,1);
    \draw [thin] (v_two) -- (v_three);
    \draw [dashed, fill=gray, fill opacity=0.2] (v_two) -- (v_one) -- (v_three);
    \draw (v_one) node [below left] {$v_1$};
    \draw (v_two) node [below right] {$v_2$};
    \draw (v_three) node [above left] {$v_3$};
    \node[] at (6.5,-0.5) {$T_2$};
\end{tikzpicture}
  
    \caption{Types of half-open unimodular simplices in $\R^2$.}
    \label{figure:threesimplices}
\end{figure}
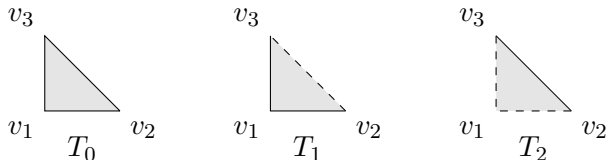

Up to unimodular transformations, there are three types of half-open unimodular simplices in $\R^2$ that we will consider; these are $T_0$, $T_1$, and $T_2$ as given in Figure~\ref{figure:threesimplices}. 

\goodbreak
\subsection{A Pick-type vector formula}

To determine the $h^1$-tensors from Theorem~\ref{prop:h_cone}, note that the Eulerian polynomial has a closed form
\begin{equation}\label{eq:A_j}
A_j(t)\ = \ \sum_{n=0}^j\sum_{i=0}^n(-1)^i\binom{j+1}{i}(n-i)^jt^n
\end{equation}
(see, e.g.,~\cite{ccd}). We then observe that $A_0(t)=1$, $A_1(t)=t$, and $A_2(t)=t^2+t$.

A comparison of coefficients of the numerator of~(\ref{eq:h_cone}) and that in Theorem~\ref{prop:h_cone} yields the formula
\begin{align*}
h_{S^{\ast}}^1(t)\ = \ \sum_{i=0}^2\OL^1(S_i)t^{i}(1-t)+\OL(S_i)t^{i+1}(v_1+v_2+v_3)
\end{align*} implying that
\begin{equation}\label{eq:hi_vector}
h_i^1(S^{\ast})\ = \ \OL^1(S_i)-\OL^1(S_{i-1})+\OL(S_{i-1})(v_1+v_2+v_3)
\end{equation}
for a half-open simplex $S^{\ast}$ where $S_i$ are defined as in Section~\ref{sec:half-open}.

By Theorem~\ref{thm:half-open decomposition}, any lattice polygon can be partitioned into unimodular transformations of half-open simplices. Therefore, to calculate $h^r$-tensors, we will need to understand the half-open parallelepipeds $\Pi_{T_0}$, $\Pi_{T_1}$, and $\Pi_{T_2}$. For ease, we provide skeletal descriptions of these here. By setting $S^{\ast}$ to $T_0$, $T_1$, and $T_2$ with the vertices given in Figure~\ref{figure:threesimplices}, we obtain:
\begin{alignat}{5}\label{eq:piped}
T_0 &\ : \ S_0\cap\Z^2=\{ 0 \};\nonumber \\ 
T_1 &\ : \ S_1\cap\Z^2=\{v_1\} ;\\
T_2 &\ : \ S_2\cap\Z^2=\{v_2+v_3\} \nonumber
\end{alignat}
\noindent where $S_i\cap\Z^2=\varnothing$ for any combination of $S_i$, $T_j$ not given.

\begin{prop}\label{h* r=1}
For any lattice polygon, we have
\[
h_P^1(t)\ = \ t\sum_{V}x+t^2\left(\sum_{E^o}(y+z) - 2\sum_{V^o}x\right)+t^3\sum_{V^o}x.
\]
\end{prop}

\goodbreak
\begin{proof}
We determine the $h^1$-tensor polynomial of all half-open unimodular simplices, up to a unimodular transformation, with vertices $v_1,v_2,v_3$. Using formula (\ref{eq:hi_vector}) together with the values given in~(\ref{eq:piped}), we obtain the following $h^1$-tensor polynomials for each $T_i$:
\begin{align*}
h_{T_0}^1(t)&\ = \ t(v_1+v_2+v_3)\\
h_{T_1}^1(t)&\ = \ tv_1+t^2(v_2+v_3)\\
h_{T_2}^1(t)&\ = \ t^2((v_1+v_2)+(v_1+v_3)-2v_1)+t^3v_1
\end{align*}
Theorem~\ref{thm:half-open decomposition} together with a careful inspection of the $h^1$-tensor polynomials of the half-open simplices yield the result.
\end{proof}

From Proposition~\ref{h* r=1}, we can deduce formulas for the Ehrhart vectors.

\begin{prop}
For any lattice polygon,
\[
\OL^1(nP)\ = \ \frac{n}{6}\left(2\sum_{V}x+4\sum_{V^o}x-\sum_{E^o}(y+z)\right)+\frac{n^2}{2}\sum_{\partial V}x+\frac{n^3}{6}\left(\sum_{\partial V}x + \sum_{E^o}(y+z)\right)
\]
\end{prop}
\begin{proof}
By definition, the Ehrhart vector polynomial equals
\[
\OL^{1}(nP)\ = \ h_0^{1}(P) {n+3 \choose 3} + h_1^{1}(P) {n+2 \choose 3}+ h_2^{1}(P) {n+1 \choose 3}+ h_3^{1}(P) {n \choose 3} \, .
\]
A substitution of values from Proposition~\ref{h* r=1} yields
\[
\OL^{1}(nP)\ = \ \frac{n^3+3n^2+2n}{6}\sum_V x+\frac{n^3-n}{6}\left(\sum_{E^o}(y+z) - 2\sum_{V^o}x\right)+\frac{n^3-3n^2+2n}{6}\sum_{\interior V}x\, .
\]
The result now follows from a quick comparison of coefficients.
\end{proof}

\goodbreak
\subsection{A Pick-type matrix formula}

We now determine the $h^2$-tensors in order to find a Pick-type formula for the discrete moment matrix.

Similar to the vector case, by comparing coefficients of the numerator of~(\ref{eq:h_cone}) and that in Theorem~\ref{prop:h_cone}, we obtain the formula
\begin{align*}
h_{S^{\ast}}^2(t)\ = \ \sum_{i=0}^2&\OL^2(S_i)t^i(1-t)^2+2(v_1+v_2+v_3)\OL^1(S_i)t^{i+1}(1-t)\\
&+(v_1^2+v_2^2+v_3^2)\OL(S_i)t^{i+1}+(v_1+v_2+v_3)^2\OL(S_i)t^{i+2}
\end{align*}
for a half-open simplex $S^{\ast}$ where $S_i$ are defined as in Section~\ref{sec:half-open}. The $h^2$-tensors of a half-open simplex are then found to be
\begin{align}\label{eq:hi_matrix}
\begin{split}
h_i^2(S^{\ast})\ = \ &\OL^2(S_i)-2\OL^2(S_{i-1})+\OL^2(S_{i-2})+2(v_1+v_2+v_3)\left(\OL^1(S_{i-1})-\OL^1(S_{i-2})\right)\\
&+(v_1^2+v_2^2+v_3^2)\OL(S_{i-1})+(v_1+v_2+v_3)^2\OL(S_{i-2})\, .
\end{split}
\end{align}

\goodbreak
\begin{prop}\label{h* r=2}
If $P$ is a lattice polygon, then
\[
h_P^2(t)\ = \ t\sum_{V}x^{2}+t^2\left(\sum_{E}(y+z)^{2} - \sum_{V}x^{2}\right)+t^3\left(\sum_{E^o}(y+z)^{2} - \sum_{V^o}x^{2}\right)+t^4\sum_{V^o}x^{2}.
\]
\end{prop}

\begin{proof}
Similar to the $h^1$-tensor polynomial, we determine the $h^2$-tensor polynomial of all half-open unimodular simplices, up to unimodular transformation. Formula~(\ref{eq:hi_matrix}) for each $T_i$ with the values from~(\ref{eq:piped}) yields the following:
\begin{align*}
h_{T_0}^2(t)&\ = \ t(v_1^2+v_2^2+v_3^2)+t^2((v_1+v_2)^2+(v_2+v_3)^2+(v_3+v_1)^2-v_1^2-v_2^2-v_3^2)\\
h_{T_1}^2(t)&\ = \ tv_1^2+t^2((v_1+v_2)^2 + (v_1+v_3)^2 - v_1^2)+t^3(v_2+v_3)^2\\
h_{T_2}^2(t)&\ = \ t^2(v_2+v_3)^2+t^3((v_1+v_2)^2 + (v_1+v_3)^2 - v_1^2)+t^4v_1^2
\end{align*}
The claim now follows from Theorem~\ref{thm:half-open decomposition}.
\end{proof}

\goodbreak
From Proposition~\ref{h* r=2}, we can now deduce formulas for the Ehrhart matrices.

\begin{prop}
Given a lattice polygon $P$, we have
\begin{align*}
\OL^2(nP)\ = \ &\frac{n}{12}\sum_{\partial E}(y-z)^2+\frac{n^2}{24}\left(12\sum _V x^2+12\sum _{V^o} x^2-\sum_{E}(y+z)^2-\sum_{E^o}(y+z)^2\right)\\
&+\frac{n^3}{12}\left(2\sum_{\partial V}x^2+\sum_{\partial E}(y+z)^{2}\right)+\frac{n^4}{24}\left(\sum_{E}(y+z)^2+\sum_{E^o}(y+z)^2\right).
\end{align*}
\end{prop}

\begin{proof}
By definition, the Ehrhart matrix polynomial equals
\[
\OL^{2}(nP)\ = \ h_0^{2}(P) {n+4 \choose 4} + h_1^{2}(P) {n+3 \choose 4}+ h_2^{2}(P) {n+2 \choose 4}+ h_3^{2}(P) {n+1\choose 4}+ h_4^{2}(P) {n\choose 4} \, .
\]
The result follows now from Proposition~\ref{h* r=2} and comparing coefficients. 
For $\OL_{1}^{2}(P)$, we further observe that
\[
\OL_{1}^{2} (P)\ = \ \frac{1}{12}\left(4\sum_{\partial V}x^2-\sum_{\partial E}(y+z)^{2}\right)\ = \ \frac{1}{12}\sum_{\partial E}(y-z)^{2} \, .
\qedhere
\]
\end{proof}

\goodbreak
\section{Positivity for $h^2$-vectors}\label{sec:positivity}

A fundamental theorem in Ehrhart theory is Stanley's Nonnegativity Theorem \cite{RS80} that states that the $h^\ast$-vector of every lattice polytope has nonnegative entries. 
While positivity of real numbers is canonically defined up to sign change, there are many different choices for higher dimensional vector spaces such as $\TT^r$; one for every pointed cone (compare, e.g., \cite{jochemkocombinatorial}). An important and well-studied cone inside the vector space of symmetric matrices is the cone of positive semidefinite matrices. 

\goodbreak
A matrix $M\in \R^{d\times d}$ is called \textbf{positive semidefinite} if $x^tMx\geq 0$ for all $x\in \R^d$. By the identification of $\TT^2$ with $\R^{d\times d}$, we call a tensor $T\in \TT^2$ \textbf{positive semidefinite} if its corresponding symmetric matrix $(T_{ij})$ is positive semidefinite. By the spectral theorem, $T$ is a \textbf{sum of squares}; more precisely, if $T$ has eigenvalues $\lambda_1,\dots,\lambda_d \geq 0$ and corresponding normalized eigenvectors $u_1,\dots,u_d$ then 
\[
(T_{ij})\ = \ \sum_{k=1}^d\lambda_ku_ku_k^t\, 
\]
which is equivalent to $T=\sum_{k=1}^d\lambda_ku_k^2\in \TT^2$. Therefore, a tensor is positive semidefinite if and only if it is a sum of squares. 

As is the case for usual Ehrhart polynomials, the coefficients of Ehrhart tensor polynomials can be negative. However, in contrast to Ehrhart polynomials, this phenomenon appears already in dimension $2$. For segments, it can be seen that the linear coefficient of the Ehrhart tensor polynomial  is $\sum _E (y-z)^2$. Furthermore, by \cite[Lemma 26]{LS} and Proposition~\ref{prop:secondcoeff}, all coefficients for line segments are positive semidefinite. The following example demonstrates negative definiteness in the plane.

\begin{ex}\label{ex:nonpositiveEhrharttensors}
    Let $P$ be the triangle spanned by vertices
    $v_1 = (0,1)^t$,
    $v_2 = (-1,-7)^t$ and
    $v_3 = (1,-4)^t$.
    The Ehrhart tensor polynomial of $P$ can be calculated to be
    \begin{align*}
        \OL^{2}(nP)=
        \begin{pmatrix} \frac{1}{2} & \frac{3}{4} \\[6pt] \frac{3}{4} &
\frac{49}{6} \end{pmatrix} n
+ \begin{pmatrix} -\frac{1}{12} & -\frac{1}{8} \\[6pt] -\frac{1}{8} &
        -\frac{23}{12}
\end{pmatrix} n^2
+ \begin{pmatrix} \frac{1}{2} & \frac{3}{4} \\[6pt] \frac{3}{4} & \frac{149}{6}
\end{pmatrix} n^3
+ \begin{pmatrix} \frac{13}{12} & \frac{13}{8} \\[6pt] \frac{13}{8} &
\frac{1079}{12}
\end{pmatrix} n^4.
    \end{align*}
We observe that the coefficient of $n^2$ is negative definite.
Lattice triangles for which this coefficient is indefinite also exist; for
example, the triangle with vertices at $(0,-4)^t$, $(0,4)^t$ and $(-1,0)^t$.
\end{ex}

Our main result is the following analogue to Stanley's Nonnegativity Theorem for the $h^2$-tensor polynomial of a lattice polygon.

\begin{thm}\label{thm:main}
The $h^2$-tensors of any lattice polygon are positive semidefinite.
\end{thm}

Before proving Theorem~\ref{thm:main}, we make a few more observations. Positive semidefiniteness of $h^2$-tensors is preserved under unimodular transformations since, from Equation~\eqref{eq:expansion} and comparing coefficients, we have 
\[  
  h^r _i (\phi P)(v,v)=h^r_i (\phi^t v, \phi ^t v)
  \]
for all $P\in \mathcal{P}(\Z^d)$, $\phi \in \GL (\Z^d)$, and $v\in \mathbb{R}^d$. However, as the next example shows, positive semidefiniteness of the $h^2$-vector is in general not preserved under translation. 

    \begin{ex} \label{ex:counter_ex_monotonicity}
      Let $S =  \conv\{v_1,v_2,v_3\}\setminus\conv\{v_2,v_3\}$ be the half-open simplex with vertices
      $v_1 =~(3,-2)^t$,
      $v_2 = (2,-2)^t$, and
      $v_3 = (2,-1)^t$.
      From the formula of the $h^2$-vector of a half-open simplex which can be found in the proof of Proposition~\ref{h* r=2},
      we obtain that
      \begin{equation*}
          h_S^2(t)\ =  \ \begin{pmatrix} 4 & -4 \\ -4 & 4
\end{pmatrix}t+\begin{pmatrix} 37 & -28 \\ -28 & 21
\end{pmatrix}t^2+\begin{pmatrix} 25 & -15 \\ -15 & 9
\end{pmatrix}t^3\, .
        \end{equation*}
That is, with a determinant of $-7$, the matrix $h^2_2(S)$ is not positive
semidefinite. However, it can be seen that the positive semidefiniteness of $h^2$-tensors is not preserved under translations. 
To illustrate, consider the translate $S-v_1$. The $h^2$-vector of the translated simplex
      \begin{equation*}
          h_{S-v_1}^2(t)\ =  \ \begin{pmatrix} 1 & 0 \\ 0 & 1 \end{pmatrix}t^2
        + \begin{pmatrix} 1 & 1 \\ 1 & 1 \end{pmatrix} t^3, \\
        \end{equation*}
has positive semidefinite coefficients.
  \end{ex}
  
  \goodbreak
Since Example~\ref{ex:counter_ex_monotonicity} shows that $h^2$-tensors of half-open polytopes can be negative, it follows that $h^2$-tensors are not monotone with respect to inclusion in contrast to the coefficients of the $h^\ast$-polynomial~\cite{RS93}. Therefore, techniques such as irrational decomposition or half-open decomposition that succesfully helped prove Stanley's Nonnegativity Theorem (see \cite{ccd,jochemkocombinatorial}) cannot immediately be applied with Theorem \ref{thm:half-open decomposition}; we will have to take a different route.

To prove Theorem \ref{thm:main}, we decompose a lattice polygon into lattice polygons with few vertices for which the $h^2$-vectors can easily be calculated. For the remainder of this article, allow a lattice polygon to always mean a full-dimensional in $\R^2$ although the argument is independent from the chosen ambient space. A \textbf{sparse decomposition} of $P\in\poly$ is a finite set $\mathcal{D}=~\{P_1,P_2,\ldots, P_m\}$ of lattice polygons such that 
\begin{enumerate}[i)]
\item $\OL(P_i)\in\{3,4\}$ for each $i\in[m]$, 
\item $P_i \cap P_j = \emptyset$ or is a common vertex of $P_i$ and $P_j$ for all $i\neq j$, and
\item $P\cap\Z^2 = \bigcup_{i=1}^m P_i\cap\Z^2$.
\end{enumerate}

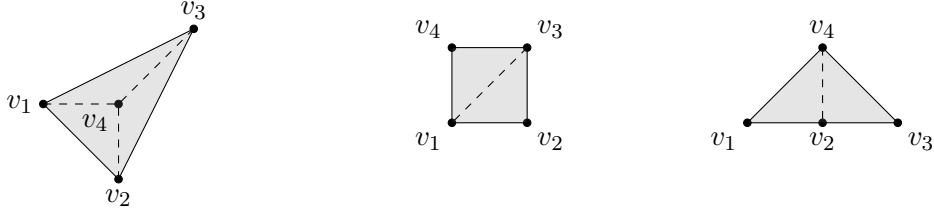
\begin{figure}
\begin{minipage}{0.3\textwidth}
    \begin{tikzpicture}
    \clip (-2.7,-1.7) rectangle (1.2,3.2);

    \coordinate (v_one) at (-1,0);
    \coordinate (v_two) at (0,-1);
    \coordinate (v_three) at (0,0);
    \coordinate (v_four) at (1,1);

    \node[draw,circle,inner sep=1pt,fill] at (v_one) {};
    \node[draw,circle,inner sep=1pt,fill] at (v_two) {};
    \node[draw,circle,inner sep=1pt,fill] at (v_three) {};
    \node[draw,circle,inner sep=1pt,fill] at (v_four) {};
    \draw [thin, fill=gray, fill opacity=0.2] (v_one) -- (v_two) -- (v_four) -- (v_one);
        \draw [thin, dashed] (v_one) -- (v_three);
                \draw [thin, dashed] (v_two) -- (v_three);
                        \draw [thin, dashed] (v_four) -- (v_three);
    \draw (v_one) node [left] {$v_1$};
    \draw (v_two) node [below] {$v_2$};
    \draw (v_three) node [below left] {$v_4$};
    \draw (v_four) node [above] {$v_3$};
\end{tikzpicture}
\end{minipage}
\begin{minipage}{0.3\textwidth}
    \begin{tikzpicture}
    \clip (-2.2,-1.2) rectangle (1.7,3.2);

    \coordinate (v_one) at (0,0);
    \coordinate (v_two) at (0,1);
    \coordinate (v_three) at (1,0);
    \coordinate (v_four) at (1,1);

    \node[draw,circle,inner sep=1pt,fill] at (v_one) {};
    \node[draw,circle,inner sep=1pt,fill] at (v_two) {};
    \node[draw,circle,inner sep=1pt,fill] at (v_three) {};
    \node[draw,circle,inner sep=1pt,fill] at (v_four) {};
    \draw [thin, fill=gray, fill opacity=0.2] (v_one) -- (v_two) -- (v_four) -- (v_three)--(v_one);
                        \draw [thin, dashed] (v_four) -- (v_one);
    \draw (v_one) node [below left] {$v_1$};
    \draw (v_two) node [above left] {$v_4$};
    \draw (v_three) node [below right] {$v_2$};
    \draw (v_four) node [above right] {$v_3$};
\end{tikzpicture}
\end{minipage}
\begin{minipage}{0.3\textwidth}
    \begin{tikzpicture}
    \clip (-2.2,-1.2) rectangle (1.7,3.2);

    \coordinate (v_one) at (-1,0);
    \coordinate (v_two) at (0,0);
    \coordinate (v_three) at (1,0);
    \coordinate (v_four) at (0,1);

    \node[draw,circle,inner sep=1pt,fill] at (v_one) {};
    \node[draw,circle,inner sep=1pt,fill] at (v_two) {};
    \node[draw,circle,inner sep=1pt,fill] at (v_three) {};
    \node[draw,circle,inner sep=1pt,fill] at (v_four) {};
    \draw [thin, fill=gray, fill opacity=0.2] (v_one) -- (v_two) -- (v_three) --(v_four) -- (v_one);
                            \draw [thin, dashed] (v_four) -- (v_two);
    \draw (v_one) node [below left] {$v_1$};
    \draw (v_two) node [below] {$v_2$};
    \draw (v_three) node [below right] {$v_3$};
    \draw (v_four) node [above] {$v_4$};

\end{tikzpicture}
  
\end{minipage}
\caption{Lattice polygons with 4 lattice points and their unimodular triangulations.}
\label{figure:polygon4points}
\end{figure}

\begin{lemma}{\cite[Section 4]{LZ11}}\label{lem:fig}
Up to unimodular transformation, there are three different lattice polygons containing exactly four lattice points. They are given in Figure \ref{figure:polygon4points}.
\end{lemma}

The following lemma ensures that every lattice polygon has a sparse decomposition.

\goodbreak
\begin{lemma} \label{ref:lemma_decomp}
  Every lattice polygon has a sparse decomposition.
\end{lemma}

\goodbreak
\begin{proof} We proceed by induction on $\OL(P)$. The statement is trivially true if $\OL(P)\in \{3,4\}$. Hence, we may assume that $\OL(P)>4$
  and choose a vector $a\in\R^2\setminus\{0\}$ such that
  $a^t v \neq a^t w$ for each $v,w \in P\cap\Z^2$ where $v\neq w$. Note
  that such an $a$ exists since $\OL(P)$ is finite. Let $P\cap\Z^2 = \{ v_1,\ldots, v_n\}$ be such that 
  \[
  a^t v_1 \ > \ a^t v_2  \ > \ \cdots \ > \ a^t v_n
  \]
  and set $ Q = \conv\{v_3,v_4,\ldots, v_n\}$.
  Then, by convexity, we obtain $Q\cap\Z^2 = P\cap\Z^2\setminus\{v_1,v_2\}$. 
  
\goodbreak  
  If $Q$ is not full-dimensional and all lattice points of $Q$ lie on a
   line, then a sparse decomposition of $P$ can
   easily be constructed. If $u_1,u_2,$ and $u_3$ are not collinear, then we can construct a sparse decomposition which is illustrated in Figure~\ref{figure:decomp_aff_dep}. Let $P_1=\conv\{u_1,u_2,u_3\}$. Then, by design, the triangle $P_1$ does not contain any other lattice point and at least one of $u_1$ or $u_2$ are visible from all points $u_4,\ldots, u_n$. Without loss of generality, assume $u_1$ is visible. Then for all $2\leq i\leq \lfloor \frac{n}{2} \rfloor -1$ define $P_i =\conv \{u_1,u_{2i},u_{2i+1}\}$, $P_{\lfloor \frac{n}{2} \rfloor} =\conv \{u_1,u_{n-2},u_n\}$ if $n$ is even, and $P_{\lfloor \frac{n}{2} \rfloor} =~\conv \{u_1,u_{n-1},u_n\}$ if $n$ is odd. Then $\{P_1,\ldots, P_{\lfloor \frac{n}{2} \rfloor}\}$ is a sparse decomposition. If $u_1,u_2,$ and $u_3$ are collinear, then a sparse decomposition can be obtained by instead setting $P_1=\conv\{u_2,u_3,u_4\}$.
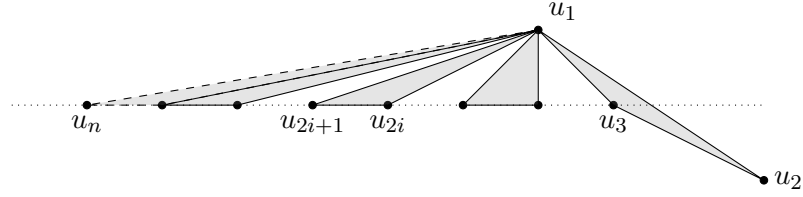
\begin{figure}
    \begin{tikzpicture}
\node[draw,circle,inner sep=1pt,fill] at (-1,0) {};
\node[draw,circle,inner sep=1pt,fill] at (0,0) {};
\node[draw,circle,inner sep=1pt,fill] at (1,0) {};
\node[draw,circle,inner sep=1pt,fill] at (2,0) {};
\node[draw,circle,inner sep=1pt,fill] at (3,0) {};
\node[draw,circle,inner sep=1pt,fill] at (4,0) {};
\node[draw,circle,inner sep=1pt,fill] at (5,0) {};
\node[draw,circle,inner sep=1pt,fill] at (6,0) {};
\node[draw,circle,inner sep=1pt,fill] at (5,1) {};          
\node[draw,circle,inner sep=1pt,fill] at (8,-1) {}; 

\draw [thin, fill=gray, fill opacity=0.2] (0,0)--(1,0)--(5,1)--(0,0);

\draw [thin, fill=gray, fill opacity=0.2] (2,0)--(3,0)--(5,1)--(2,0);
\draw [thin, fill=gray, fill opacity=0.2] (4,0)--(5,0)--(5,1)--(4,0);
\draw [thin, fill=gray, fill opacity=0.2] (6,0)--(5,1)--(8,-1)--(6,0);
\draw [thin, dashed, fill=gray, fill opacity=0.2] (5,1)--(-1,0)--(0,0)--(5,1);
\draw (5,1) node [above right] {$u_1$};
\draw (8,-1) node [right] {$u_2$};
\draw (6,0) node [below] {$u_3$};
\draw (3,0) node [below] {$u_{2i}$};
\draw (2,0) node [below] {$u_{2i+1}$};
\draw (-1,0) node [below] {$u_n$};
\draw [thin, dotted] (-2,0)--(8,0);
\end{tikzpicture}
    \caption{Sparse decomposition of $P$ for the case of a collinear $Q$.}
    \label{figure:decomp_aff_dep}
\end{figure}

\goodbreak
Suppose $Q$ is full-dimensional. Then, by the induction hypothesis, there is a sparse decomposition $\mathcal{D}_Q$ of $Q$. Let $i$ be the smallest index such that the points $u_1,u_2,u_{i}$ do not lie on a common straight line. By construction, the simplex $S=\conv (u_1,u_2,u_{i})$ contains no other lattice points and, thus, $\mathcal{D}_Q\cup\{S\}$ is a sparse decomposition of $P$.
\end{proof}

\goodbreak
\begin{lemma} \label{ref:lemma_h2_small_polytope}
If $P\in\polyt$ is a lattice polygon containing exactly three or four lattice points, then $h_2^2(P)$ is positive semidefinite.
\end{lemma}

\begin{proof}
If $\OL(P)=3$, then $P=\conv(v_1,v_2,v_3)$ is a unimodular lattice simplex and the statement follows from Proposition~\ref{h* r=2} as
\[
h_2^2(P)=(v_1+v_2)^2+(v_1+v_3)^2+(v_2+v_3)^2-v_1^2-v_2^2-v_3^2=(v_1+v_2+v_3)^2.
\]

Suppose $\OL(P)=4$. We have to distinguish between the three possible cases, up to unimodular transformation, given in Figure~\ref{figure:polygon4points}. First, if $P$ contains one interior lattice point $v_4$ and vertices $v_1, v_2,v_3$, then we have $v_4=\tfrac{1}{3}(v_1+v_2+v_3)$ and Proposition~\ref{h* r=2} implies that
  \begin{align*}
    h_2^2(P) &= (v_1+v_2)^{2}  + (v_1+v_3)^{2} + (v_2+v_3)^{2} + 
    (v_1+v_4)^{2} + (v_2+v_4)^{2} + (v_3+v_4)^{2} \\
    &\quad - v_1^{2}- v_2^{2}- v_3^{2}- v_4^{2} \\
    &=(v_1+v_2)^{2} + (v_1+v_3)^{2} + (v_2+v_3)^{2} + 2v_4^2+2v_4(v_1+v_2+v_3)\\
    &= (v_1+v_2)^{2} + (v_1+v_3)^{2} + (v_2+v_3)^{2}  +
      \tfrac{8}{9}(v_1+v_2+v_3)^{2}.
  \end{align*}

Next, if $P$ is a parallelepiped, then $v_1+v_3 = v_2+v_4$ and thus
  \begin{align*}
    h_2^2(P) &= (v_1+v_2)^{2} + (v_2+v_3)^{2} + (v_3+v_4)^{2} + (v_1+v_4)^{2}\\
        &+\tfrac{1}{2}(v_1+v_3)^{2} +\tfrac{1}{2}(v_2+v_4)^{2}
    - v_1^{2}- v_2^{2}- v_3^{2}- v_4^{2} \\
    &=\tfrac{1}{2}(v_1+v_2+v_3+v_4)^{2}+\tfrac{1}{2}(v_1+v_2)^{2}
      +\tfrac{1}{2}(v_2+v_3)^{2}+\tfrac{1}{2}(v_3+v_4)^{2}+\tfrac{1}{2}(v_1+v_4)^{2}.
  \end{align*}

Finally, if $P$ has three vertices and no interior lattice point, then one lattice point of $P$, say $v_2$ as in Figure~\ref{figure:polygon4points}, lies in the relative interior of the edge given by the vertices $v_1$ and $v_3$ implying that $v_2=\frac{1}{2}(v_1+v_3)$. In this case, we obtain
  \begin{align*}
    h_2^2(P) &= (v_1+v_2)^{2} + (v_2+v_3)^{2} + (v_3+v_4)^{2} + (v_1+v_4)^{2}
    +(v_2+v_4)^{2}
    - v_1^{2}- v_2^{2}- v_3^{2}- v_4^{2}\\
    &=\tfrac{5}{2}v_1^2+\tfrac{5}{2}v_3^2+2v_4^2+3v_1v_4+3v_3v_4+3v_1v_3\\
      &=\tfrac{3}{2}(v_1+v_3+v_4)^{2} +v_1^{2} +v_3^{2} +\tfrac{1}{2}v_4^{2}.
    \qedhere
  \end{align*}
\end{proof}

\goodbreak
We will need the following geometric observation in our proof of Theorem~\ref{thm:main}.

\begin{lemma} \label{ref:lemma_planar_cond}
Let $P\in\polyt$ and $v$ be a lattice point in the relative interior of $P$. Then at least one of the following two statements is true:
  \begin{enumerate}[(i)]
      \item \label{ref:lemma_two_cond_i}
       $v=\tfrac{1}{2}(v_1+v_2)$ for lattice points $v_1, v_2 \in P$ such that $v_1\neq v_2$;
      \item \label{ref:lemma_two_cond_ii}
        $v=\tfrac{1}{3}(v_1+v_2+v_3)$ for pairwise disjoint lattice points $v_1, v_2, v_3 \in P$.
  \end{enumerate}
\end{lemma}

\begin{proof}
If $v$ is contained in a segment formed by two lattice points in $P$, then $v$ is easily seen to be of the form given in \emph{(i)}.

Therefore, we may assume that $v$ is not contained in any line segment formed by lattice points in $P$. By Caratheodory's Theorem~(see, e.g., \cite{Schneider:CB2}), there are lattice points $v_1,v_2,v_3\in P$ such that  $v$ is contained in the simplex formed by $v_1,v_2,$ and $v_3$. If $v,v_1,v_2,v_3$ are the only lattice points in the simplex, then condition \emph{(ii)} follows from Lemma~\ref{lem:fig}. Otherwise, there is a lattice point $u\in\conv\{v_1,v_2,v_3\}\setminus\{v,v_1,v_2,v_3\}$ and, consequently, $v$ must be contained in one of
  the three lattice simplices
  \begin{align*}
    S_1 = \conv\{v_2,v_3,u\}, \quad 
    S_2 = \conv\{v_1,v_3,u\},\quad
    S_3 = \conv\{v_1,v_2,u\}.
  \end{align*}
Without loss of generality, let $v\in S_1\subsetneq \conv\{v_1,v_2,v_3\}$. By reiteration of the above procedure, each time with a replacement of $v_1$ by $u$, we eventually find affinely independent $v_1, v_2, v_3$ such that
  $\{v,v_1,v_2,v_3\} = \conv\{v_1,v_2,v_3\}\cap\Z^2$ and condition \emph{(ii)} follows again from Lemma~\ref{lem:fig}. 
  \end{proof}
  
We are now equipped to give the proof of our nonnegativity theorem.

\begin{proof}[Proof of Theorem~\ref{thm:main}]
From Proposition~\ref{h* r=2}, it immediately follows that $h_0^{2}(P),h_1^{ 2}(P)$, and $h_4^{2}(P)$ are sums of squares.

Let $\mathcal{D}=\{P_1,P_2,\ldots,P_m\}$ be a sparse decomposition of $P$ which exists by Lemma~\ref{ref:lemma_decomp} and let $\mathcal{S}$ be some triangulation of $\cup_{i=1}^mP_i$. Observe that the closure of $P \setminus  (P_1\cup\dots\cup P_m)$ is a union of not necessarily convex lattice polygons and any triangulation of $\cup_{i=1}^mP_i$ can be extended to a triangulation in $P$. Let $\mathcal{T}$ be a triangulation of $P$ such that $\mathcal{S}\subseteq\mathcal{T}$. Let $G=(V,E)$ be the edge graph of $\mathcal{T}$ and $G'=(V',E')$ be the edge graph of $\mathcal{S}$. For every $x\in V$, we define $\alpha_x = |\{i \in [m]: x\in P_i\}|$. Note that $\alpha_x \geq 1$ for all $x\in V$ since $\mathcal{D}$ is a
sparse decomposition.
Proposition~\ref{h* r=2} then implies that
  \begin{align*}
      h_2^{2}(P) &= \sum_{E}(y+z)^{ 2} - \sum_{V}x^{ 2}\\
      &= \sum_{E'} (y+z)^{ 2} 
                - \sum_{V}\alpha_x \ x^{ 2}
          +  \sum_{E\setminus E'}
                (y+z)^2
          - \sum_{V}(1-\alpha_x) \ x^{ 2} \\
          &=  \sum_{i=1}^m h_2^{2}(P_i)
          +  \sum_{E\setminus E'}
                (y+z)^2
          + \sum_{V}(\alpha_x-1) \ x^{ 2} \, ,
  \end{align*}
and therefore, by Lemma~\ref{ref:lemma_h2_small_polytope}, $h_2^2(P)$ is a sum of squares.

  We have left to show that $h_3^{2}(P)$ is also a sum of squares.
  For every $v\in V$, we define
  $N(v)=\{u\in V: \{u,v\}\in E\}$
  to be the set of vertices adjacent to $v$ in $G$. Let $E_1\subseteq \interior E$ be the set of edges that have exactly one endpoint on the boundary of $P$ and $E_2\subseteq\interior E$ be the set of edges with both endpoints on the boundary of $P$ but relative interior in $\interior P$. By Proposition~\ref{h* r=2}, we obtain
  \begin{align*}
      h_3^{2}(P) &= \sum_{\interior E}
      (y+z)^{ 2} - \sum_{\interior{V}}x^{ 2}\\
      &= \sum_{\substack{v \in \interior{V}}}\left(\sum_{u \in N(v)}\left(\frac{1}{2}(v+u)^{ 2}\right) - v^{ 2}\right)  + \sum _{E_1} \frac{1}{2}(y+z)^2 + \sum _{E_2} (y+z)^2 \, .
  \end{align*}
  It is thus sufficient to show that
  \begin{equation*}
    a(v) := \sum_{u \in N(v)}\left(\frac{1}{2}(v+u)^{ 2} - v^{ 2}\right)
  \end{equation*}
  is a sum of squares for all $v\in \interior{V}$. In view of Lemma~\ref{ref:lemma_planar_cond}, we
  distinguish two cases. First, suppose that there are 
  $v_1, v_2 \in V \setminus \{v\}$ such that $v=\tfrac{1}{2}(v_1+v_2)$.
  Then
  \begin{align*}
    a(v)&=\tfrac{1}{2}(v + v_1)^{ 2} +\tfrac{1}{2}(v + v_2)^{ 2} - v^{ 2} + \sum _{u\in N(v)\setminus \{v_1,v_2\}} \tfrac{1}{2}(v+u)^{ 2}\\
    &= \tfrac{1}{2}(v_1+v_2)^{ 2} + \tfrac{1}{2} v_1^{ 2} + \tfrac{1}{2} v_2^{2}+ \sum _{u\in N(v)\setminus \{v_1,v_2\}} \tfrac{1}{2}(v+u)^{ 2} \, .
  \end{align*}
  In the second case, there exist pairwise disjoint
  $v_1, v_2, v_3 \in V \setminus \{v\}$ such that $v=\tfrac{1}{3}(v_1+v_2+v_3)$.
  Therefore
  \begin{align*}
    a(v)&=\tfrac{1}{2}(v + v_1)^{ 2}
        +\tfrac{1}{2}(v + v_2)^{ 2}
        +\tfrac{1}{2}(v+v_3)^{ 2} - v^{ 2}+\sum _{u\in N(v)\setminus \{v_1,v_2,v_3\}} \tfrac{1}{2}(v+u)^{ 2}\\
    &= \tfrac{7}{18}(v_1+v_2+v_3)^{ 2}
        + \tfrac{1}{2} v_1^{ 2}
        + \tfrac{1}{2} v_2^{ 2}
        + \tfrac{1}{2} v_3^{ 2}+\sum _{u\in N(v)\setminus \{v_1,v_2,v_3\}} \tfrac{1}{2}(v+u)^{ 2} \, .\qedhere
  \end{align*}
\end{proof}

\goodbreak
\section{Further results and outlook}\label{sec:further}

It is natural to ask whether Theorem~\ref{thm:main} holds true in higher dimensions. Using the software package polymake~\cite{polymake:2017,polymake:2000} we have calculated the $h^2$-tensor polynomials of several hundred randomly generated polytopes in dimension $3$ and $4$. Based on these computational results, we offer the following conjecture.
\begin{conj}\label{conj:pos} 
For $d\geq 1$, the coefficients of the $h^2$-tensor polynomial of any of a lattice polytope in $\R^d$ are positive semidefinite.
\end{conj}
For our proof of Theorem~\ref{thm:main}, it was crucial that every lattice polygon has a unimodular triangulation. Since this no longer holds true in general for higher dimensional polytopes, a proof of Conjecture~\ref{conj:pos} would need to be conceptually different. 

Finding inequalities among the coefficients of the $h^\ast$-polynomial of a
lattice polytope, beyond Stanley's Nonnegativity Theorem, is currently of
great
interest in Ehrhart theory. The ultimate goal is a classification of all
possible $h^\ast$-polynomials: a classification of all
$h^\ast$-polynomials
of degree $2$ can be found in \cite[Proposition 1.10]{HM09}. Another fundamental
inequality
is due to Hibi~\cite{MR1275662} who proved that $h_i(P)-h_1(P)\geq 0$ for all
$1\leq i<d$ and full-dimensional lattice polytopes that have an interior
lattice
point. Calculations with polymake again suggest that there
might be a version for matrices motivating the following conjecture.
\begin{conj}\label{conj:hibi}
Let $P$ be a lattice polytope containing a lattice point in its interior.
Then
the matrices $h^2_i(P)-h^2_1(P)$ for $1\leq i < \dim(P)+2$ are positive
semidefinite.
\end{conj}
In recent years, additional inequalities for the coefficients of the
$h^\ast$-polynomial have been shown (see e.g.
~\cite{Athanasiadis04,Stapledon09,Stapledon16})
which raises the question as to whether there are analogous results for Ehrhart tensors.

\begin{ques}\label{ques:1}
Which known inequalities among the coefficients of the $h^\ast$-polynomial
of a lattice polytope can be generalized to $h^r$-tensor polynomials of
higher rank?
\end{ques}

An answer would depend on the notion of positivity that is chosen. A natural choice for higher rank $h^r$-tensors, extending positive semidefiniteness of matrices, is to define $T\in\TT^r$ to be positive semidefinite if and only if $T(v,\dots,v)\geq 0$ for all $v\in\R^d$. However, assuming this definition of positivity, there can not be any inequalities that are valid for all polytopes if the rank $r$ is odd since $T(v,\ldots, v) = (-1)^rT(-v,\ldots, -v)$.

In the case that $r$ is even, we are able to extend  another classical result, namely Hibi's  Palindromic Theorem~\cite{Hibi91} characterizing reflexive polytopes. A lattice polytope $P\in \mathcal{P}(Z^d)$ is called \textbf{reflexive} if 
\[
P = \{x\in \R^d \colon Ax \leq 1\}
\]
where $A\in Z^{d\times d}$ is an integral matrix.

\goodbreak
\begin{thm}[Hibi~\cite{Hibi91}]\label{thm:hibioriginal}
A polytope $P\in\poly$ is reflexive if and only if $h_i^{\ast}(P)=h_{d-i}^{\ast}(P)$ for all $0\leq i\leq d$.
\end{thm}

 A crucial step in the proof of Theorem~ \ref{thm:hibioriginal} is to observe that a polytope $P$ is reflexive if and only if
\[
nP\cap \Z^d =(n+1)\interior P \cap \Z^d
\]
for all $n\in\N$ (see \cite{ccd}). We use this fact to give the following generalization.
\begin{prop}\label{thm:hibiext}
Let $r\in\N$ be even and $P\in\poly$ be a lattice polytope that contains the origin in its relative interior. The polytope $P$ is reflexive if and only if $h_i^{r}=h_{d+r-i}^{r}$ for all $0\leq i\leq d+r$.
\end{prop}

\begin{proof}
By Theorem~\ref{thm:tensorreciprocity} and comparing coefficients in equation~\eqref{eq:expansion}, it follows that the assertion $h_i^{r}(P)=h_{d+r-i}^{r}(P)$ is equivalent to $\OL^{r}((n-1)P)=\OL^{r}(n\interior P)$ for all integers $n$. 

If $P$ is a reflexive polytope, then $\OL^{r}((n-1)P)=\OL^{r}(n\interior P)$ for all integers $n$ since, as given above, we have $(n-1)P\cap \Z^d = n\interior P \cap \Z^d$.

Now assume that $P$ is not reflexive. Then there exists an $n\in\N$ such that
\[
(n-1)P\cap \mathbb{Z}^d \subsetneq n\interior P \cap \mathbb{Z}^d \, .
\]
Therefore, for any $v\in \mathbb{R}^d\setminus \{0\}$, we obtain
\[
\sum _{x\in (n-1)P\cap \mathbb{Z}^d} (x^tv)^r  \ < \ \sum _{x\in n\interior P\cap \mathbb{Z}^d} (x^tv)^r \, 
\]
and, in particular, $\OL^{r}((n-1)P)\not =\OL^{r}(n\interior P)$ completing the proof.
\end{proof}

Note that the proof of Proposition~\ref{thm:hibiext} shows that for odd rank $r$ palindromicity of the $h^r$-tensor polynomial of a reflexive polynomial is still necessary, but not sufficient, since all centrally symmetric polytopes have a palindromic $h^r$-tensor polynomial; namely the constant zero polynomial.

\section*{Acknowledgements} 
Katharina Jochemko was partially supported by the Knut and Alice Wallenberg Foundation. Laura Silverstein was supported by the Austrian Science Fund (FWF) Projects P25515-N25 and I3017-N35.

  \bibliographystyle{siam}
\bibliography{Paper}
  
\end{document}